\documentclass[12pt]{amsart}
\usepackage{amsmath}
\usepackage{amsthm}
\usepackage{amssymb}
\usepackage[hidelinks]{hyperref}

\DeclareMathOperator{\rank}{\mathrm{rank}}
\DeclareMathOperator{\gr}{\mathrm{grad}}
\DeclareMathOperator{\cu}{\mathrm{curl}}

\newcommand{\abs}[1]{\vert #1\vert}  
\newcommand{\ov}{\overline}
\newcommand{\dif}{\mathrm{d}}
\newcommand{\ii}{\mathrm{i}}
\newcommand{\di}{\mathrm{div}}
\newcommand{\tr}{\mathrm{trace}}
\newcommand{\Vol}{\mathrm{Vol}}

\newcommand{\CC}{\mathbb{C}}
\newcommand{\Cc}{\mathfrak{C}}
\newcommand{\HH}{\mathcal{H}}

\newcommand{\VV}{\mathcal{V}}
\newcommand{\Ee}{\mathcal{E}} 
 
\newcommand{\RR}{\mathbb{R}}   
\newcommand{\ZZ}{\mathbb{Z}} 

\newcommand{\Ss}{\mathbb{S}}
\newtheorem{pr}{Proposition}
\newtheorem{co}{Corollary}
\newtheorem{lm}{Lemma}

\newtheorem{de}{Definition}    
\newtheorem{re}{Remark}
\newtheorem{ex}{Example}

\begin{document}
\title[Steady flows and Faddeev-Skyrme model]{Steady Euler flows and the Faddeev-Skyrme model with mass term}

\author{Radu Slobodeanu}
\address{Faculty of Physics, University of Bucharest, P.O. Box Mg-11, RO--077125 Bucharest-M\u agurele, Romania}
%\address{Institute of Mathematics, University of Neuch\^atel, 11 rue Emile Argand, 2000 Neuch\^atel, Switzerland.}
\email{radualexandru.slobodeanu@g.unibuc.ro}

\thanks{I thank M. Speight for insightful discussions on the Faddeev-Skyrme model and valuable comments on an earlier version of the manuscript. I also acknowledge highly useful  remarks and suggestions from D. Peralta-Salas.}



\date{\today}

\begin{abstract}
We point out a duality between integrable (in an appropriate sense) steady incompressible Euler flows and the solutions of the strongly coupled Faddeev-Skyrme sigma model with a potential term. We supplement this result with various applications and several explicit classical solutions.
\end{abstract}

\maketitle
\section{Introduction}
Effective field theory approach has allowed a rich infusion of techniques and new insights into relativistic hydrodynamics (see e.g. \cite{nico, slob} and references therein). The field is roughly speaking a submersion whose fibres are spanned by the velocity of the fluid and the action functional is constructed by requiring invariance by volume-preserving diffeomorphisms of the codomain (\cite{cristo}). Surprisingly, this variational approach is less obvious for the non-relativistic case, where however other variational formulations (\cite{arn}) are known, specifically for steady-state problem. In this paper we implement the effective field formalism for steady incompressible Euler fluids which in this way turn out to be dual to field solutions of a sigma model with quartic kinetic term and potential. This model correspond precisely to the strong coupling limit (supplemented with a mass term) of a sigma model proposed by Faddeev \cite{fad0} and renewed starting with the numerical evidence in \cite{fad}. Initially designed as an extrapolation of Skyrme's soliton idea in nuclear physics (\cite{sky}) to knot-like structures, Faddeev model acquires in the present context a remarkable physical interpretation in terms of fluid flows. The correspondence between fluids and sigma model solutions (Proposition \ref{main}) is illustrated by explicit examples and, among other possible applications, we derive the non-existence (Proposition \ref{noexist}) of certain axially symmetric field solutions, by exploiting an analogous result \cite{jiu} for steady flows on $\RR^3$ with asymptotic boundary conditions.

The paper is structured as following. In the next section we collect the relevant facts about Euler equations in fluid mechanics, extended to Riemannian 3-manifolds, and in $\S 3$ we present a variational problem related to the Faddeev-Skyrme model. In $\S 4$ we give the main result (Proposition \ref{main}) relating stationary fluid solutions with Faddeev-Skyrme fields, and we derive some extensions and applications. The last section ($\S 5$) is dedicated to  explicit examples illustrating this correspondence.

\section{Fluids in equilibrium}
A fluid moving in (a domain of) $\RR^3$ is described by the following time dependent quantities: the velocity vector $v(x,t)$, the pressure $p(x,t)$ and the density $\rho(x,t)$ scalars. In the absence of energy dissipation the fundamental equations of fluid dynamics are (\cite{land})
\begin{equation}\label{euler}
\left\{
\begin{array}{ccc}
\displaystyle \partial_t v + (v \cdot \nabla)v &=& - \tfrac{1}{\rho} \gr p \quad (\text{Euler's \ equation})\\[3mm]
\partial_t \rho + \di (\rho v) &=& 0 \qquad (\text{equation \ of \ continuity})
\end{array}
\right. ,
\end{equation}
Fluids for which these equations hold are called \textit{inviscid}, or \textit{ideal}. Here we are concerned with \textit{steady} fluids for which $\partial_t v=0$, so that the Euler equation simplifies to 
$(v \cdot \nabla)v = - \tfrac{1}{\rho} \gr p$, where $\tfrac{1}{\rho} \gr p=\gr w$, and $w$ is the \textit{enthalpy} function. By taking the scalar product with $v$ we obtain that the following quantity
$$P = w+\tfrac{1}{2}\abs{v}^2,$$
called the \textit{Bernoulli function}, is conserved along the flow.

We further assume that the fluid is \textit{incompressible}, i.e. $\rho \equiv 1$. Then an \textit{ideal steady incompressible fluid} flow $v$ is a solution of the equations
\begin{equation}\label{euler0}
\left\{
\begin{array}{ccc}
(v \cdot \nabla)v &=& - \gr p \\[2mm]
\di v &=& 0
\end{array}
\right.
\end{equation}

for some pressure function $p$, or, equivalently,

\begin{equation}\label{mhd}
\left\{
\begin{array}{ccc}
v \times \cu v &=& \gr P \\[2mm]
\di v &=& 0
\end{array}
\right.
\end{equation}
for the corresponding Bernoulli function $P = p+\tfrac{1}{2}\abs{v}^2$. For brevity we shall call such solutions \textit{steady Euler flows}. 

Notice that \eqref{mhd} coincide with the \textit{ideal magnetohydrodynamic equilibrium} equations via the substitutions $v \approx B$ (magnetic field) and $P \approx -p$ (hydrostatic pressure).

Recall that $\cu v$ is called \textit{vorticity field} and that the first equation in \eqref{mhd} implies (or, in a simply connected setting, is equivalent to) $[\cu v, v]=0$, that is the two vector fields commute. If $v$ and $\cu v$ are not everywhere collinear, then they give rise to a 2-dimensional foliation (given by tori or annuli representing the regular level surfaces of the Bernoulli function $P$) of the domain. This statement is further refined by the celebrated Arnold structural theorem \cite[Ch.II, $\S$1]{arn}. 

Finally let us mention the fundamental fact that steady fluids equations admit a divergence formulation:
\begin{equation}\label{fluidstress}
\di (\rho v^\flat \otimes v^\flat + p g)=0
\end{equation}
where $g$ is the (euclidean) metric and the divergence free quantity is the \textit{momentum-flux tensor}, $T$. For further use, notice the following equivalent form (for $\rho \equiv 1$) of this tensor at any point where $v\neq 0$: 
$T= \tfrac{1}{2}\abs{v}^2(g^\VV - g^\HH)+Pg$, where $\VV=\mathrm{Span}\{v\}$, $\HH=\VV^\perp$ and superscripts indicate the restriction to the corresponding subspace (so that $g=g^\VV + g^\HH$).

\subsection{Fluids on a manifold} Let $(M, g)$ be an oriented Riemannian 3-manifold with Levi-Civita connection $\nabla$.  We say that a tangent vector field $V$ defines a steady Euler flow on $M$ if $\nabla_V V=-\gr p$ (for some function $p$) and $\di V=0$, that reduce to Equations \eqref{euler0} in the Euclidean case. Around a point where $V$ is not zero, considering $U=\tfrac{1}{\abs{V}} V$, we notice that we can rewrite these equations as follows:
\begin{equation}\label{euler+}
\left\{
\begin{array}{ccc}
\abs{V}^2\left[\nabla_U U - \gr^{\HH} \left(\ln \abs{V}\right)\right] &=& - \gr P \\[2mm]
\di U + U(\ln \abs{V}) &=& 0
\end{array}
\right. 
\end{equation}
where $\gr^\HH$ denotes the component of the gradient orthogonal to $V$ and the Bernoulli function $P$ has the same definition as before.

Denote by $\nu_g$ the associated volume form of $(M,g)$. The \textit{vorticity field} $\cu V$ is defined by
$$
\imath_{\cu V} \nu_g= \dif V^\flat,
$$
where $\imath$ denotes the interior product, or, equivalently, by $\cu V = (* \dif V^\flat)^\sharp$, where $*$ is the Hodge star operator, $\flat$  is the isomorphism sending vectors to dual 1-forms, induced by $g$, and $\sharp$ the inverse of $\flat$. Thus Equation \eqref{mhd} and the subsequent discussion extend to this general setting.

\subsection{Beltrami fields} Recall that an important class of steady Euler flows is provided by \textit{Beltrami fields}, that are divergence-free vector fields satisfying $\cu V = f V$, for some function $f$ on $M$. If $f \equiv 0$, then they are called \textit{potential fields}, and otherwise \textit{rotational Beltrami fields}. When dealing with rotational Beltrami fields, we further distinguish between \textit{linear} or \textit{strong} ($f \equiv cst.$) and \textit{non-linear} ($f\neq cst.$) fields.

Smooth, non-singular rotational Beltrami fields enjoy the remarkable property (\cite{etn}) of being at the same time Reeb-like vector fields for some contact form on $M$ (see \cite{ble} for the contact geometry background).

\section{Strongly coupled Faddeev-Skyrme sigma model with mass term}
In this section we introduce a variational problem related to the Faddeev-Skyrme model \cite{fad0, fad}. For further use, we give explicit expressions for its stress-energy tensor and for the corresponding Euler-Lagrange equations.

Let $(M^m, g)$ and $(N^n, h)$ be Riemannian manifolds ($m\geq n$) and $\varphi: M \to N$ a differentiable mapping. Let $M_\varphi = \{ x \in M : \rank \, \dif \varphi_x < n \}$ be the \textit{critical set} of $\varphi$ (i.e. the set of \textit{critical points}); on $M \setminus M_\varphi$ (i.e. the set of \textit{regular points}),  $\VV = \ker \dif \varphi$ will denote the \textit{vertical} distribution (which is tangent to the fibres, so integrable) and $\HH = \VV^\perp$, the \textit{horizontal} distribution of $\varphi$. The map $\varphi$ is called {\it almost submersion} if it has maximal rank on a dense subset of $M$. Recall also that $\Cc_\varphi = \dif \varphi^t \circ \dif \varphi \in \mathrm{End}(TM)$ is the \textit{Cauchy-Green} (or \textit{strain}) tensor of $\varphi$.

\begin{de}
Let $\ov{P}: N \to \RR_+$ be a non-negative function on $N$. For every map $\varphi: M \to N$ the $\sigma_2$-\emph{energy with potential} over a compact domain $K$ in $M$ is defined by
\begin{equation}\label{si2pot}
\Ee_{\sigma_2 , P}(\varphi, K)=\frac{1}{2}\int_K \{\abs{\wedge^2 \dif \varphi}^2 + 2 \ov{P} \circ \varphi \}\nu_g ,
\end{equation}
A map $\varphi: M \to N$ is called $\sigma_2$-\emph{critical with potential} $\ov P$ if for every compact domain $K$ in $M$ and for any variation $\{\varphi_s\}_{s \in (-\epsilon, \epsilon)}$ supported in $K$, of $\varphi=\varphi_0$, we have
\begin{equation*}
%\begin{split}
\frac{d}{d s} \Bigl\lvert _{s=0}\Ee_{\sigma_2 , P}(\varphi_s, K)=0.
%\end{split}
\end{equation*} 
\end{de}

In this paper we shall focus on \textit{classical solutions} of this variational problem, so we consider only mappings $\varphi$ (and potential functions $\ov P$) having a sufficient regularity for the associated Euler-Lagrange equations (see below) to be well-defined. Thus $\varphi$ is requested to be of class $C^2$ and $\ov P$ of class $C^1$.

Notice that $\abs{\wedge^2 \dif \varphi}^2=\sigma_2(\varphi)=\sum_{i<j}\lambda_i^2  \lambda_j^2$, the second elementary symmetric polynomial in the eigenvalues of the strain tensor of $\varphi$. Thus $\ov P \equiv cst.$ leads us to the "pure" $\sigma_2$-variational problem introduced in \cite{eell} and discussed e.g. in \cite{sac, slo, cri}.

Using the results in \cite{slo, cri}, it is immediate to obtain the Euler-Lagrange equations associated to the $\sigma_2$-energy with potential:

\begin{pr} \label{sig2pot}
A $C^2$ map $\varphi: M \to N$ is $\sigma_2$-critical with potential $\ov{P}$ if and only if it satisfies the Euler-Lagrange equations
$$
\tau_{\sigma_2}(\varphi) - (\gr \ov{P}) \circ \varphi = 0,
$$
where $\tau_{\sigma_2}(\varphi)$ is the $\sigma_2$-\emph{tension field} of $\varphi$ given by:
$$\tau_{\sigma_2}(\varphi) = \tr \nabla(\abs{\dif \varphi}^2 \dif \varphi-\dif \varphi \circ \mathfrak{C}_\varphi).$$
\end{pr}

\medskip

In the following we shall mainly consider mappings $\varphi: (M^3, g) \to (N^2, h)$ with one-dimensional fibres into a compact Riemann surface with orthogonal complex structure $J$ and area (or K\"ahler) 2-form $\omega(\cdot, \cdot)=h(\cdot, J\cdot)$. With $M$ either euclidean $\RR^3$ or compact, this is the case typically occurring in the \textit{strongly coupled Faddeev-Skyrme model with potential} \cite{ada, fos}. Since a quadratic potential adds a mass term to the linearization of the model about the vacuum, the potential is also called \textit{mass term}.

\begin{re}
As $\dim N =2$, an \emph{equivalent} formulation of the action functional \eqref{si2pot} in terms of an area $2$-form on $N$ is 
\begin{equation*}
\Ee_{\sigma_2, P}(\varphi)=\frac{1}{2}\int_M \{\abs{\varphi^* \omega}^2 + 2 \ov{P} \circ \varphi \}\nu_g.
\end{equation*}
The associated Euler-Lagrange equations become $($\cite[Remark 2.11]{sve}$)$
\begin{equation}\label{om2poteqs}
J\dif \varphi\big((\delta\varphi^* \omega)^\sharp \big) + (\gr \ov{P}) \circ \varphi = 0.
\end{equation}
\end{re}

Since the rank cannot drop locally, at any regular point one can define the \textit{mean curvature} of the fibres / of $\VV$, by $\mu^\VV=\nabla_U U$, with $U$ being a local unit vertical vector field, and we introduce the vector field
\begin{equation}
T_\varphi = \mu^{\VV}-\gr^{\HH}(\ln \abs{\lambda_{1}\lambda_{2}}),
\end{equation}
where $\gr^{\HH}$ is the horizontal part of the gradient, and $\lambda_{1}^2$, $\lambda_{2}^2$ are the non-zero eigenvalues of $\Cc_\varphi$ (or equivalently of $\varphi^*h$ with respect to $g$).
 
\begin{lm}\label{taulem}
The $\sigma_2$-tension field of a mapping $\varphi: (M^3, g) \to (N^2, h)$ at any regular point is given by
\begin{equation}\label{taueigen}
\tau_{\sigma_2}(\varphi)=-\dif \varphi\big(\lambda_{2}^{2}g(T_\varphi, E_1)E_1 + \lambda_{1}^{2}g(T_\varphi, E_2)E_2\big),
\end{equation}
where $E_1$, $E_2$ are unit (orthogonal) eigenvectors of $\Cc_\varphi$ corresponding to the non-zero eigenvalues $\lambda_1^2$ and $\lambda_2^2$, respectively.
\end{lm}

\begin{proof} According to \cite{cri}, we have the identity
\begin{equation}\label{ide}
h(\tau_{\sigma_2}(\varphi), \dif \varphi)=-\di S_{\sigma_2}({\varphi}) ,
\end{equation}
where $S_{\sigma_2}(\varphi)=\frac{1}{2}\sigma_2(\varphi)g - \varphi^* h \circ \chi_{1}(\varphi)$ is the $\sigma_2$-\textit{stress-energy tensor} defined in terms of the first Newton tensor:
\begin{equation*}
\chi_1(\varphi) = \abs{\dif \varphi}^2 Id_{TM} - \Cc_{\varphi}.
\end{equation*}
We can check that in our case, at any regular point $x \in M$, 
\begin{equation}\label{stress2}
S_{\sigma_2}(\varphi)=\tfrac{1}{2}\lambda_{1}^2\lambda_{2}^2(g^\VV -g^\HH).
\end{equation}
so that $(\di S_{\sigma_2}(\varphi))(X)=g(\lambda_{1}^2\lambda_{2}^2 \nabla_U U - \tfrac{1}{2}\gr^\HH (\lambda_{1}^2\lambda_{2}^2), \ X)$ for any horizontal vector field $X$.
At $x$, let $E_1$, $E_2$ be unit (orthogonal) eigenvectors of $\Cc_\varphi$ corresponding to the non-zero eigenvalues $\lambda_1^2$ and $\lambda_2^2$, respectively. Since $\{ \frac{1}{\lambda_{1}} \dif \varphi(E_1) , \frac{1}{\lambda_{2}} \dif \varphi(E_2) \}$ form an orthonormal basis of $T_{\varphi(x)}N$, developing $\tau_{\sigma_2}(\varphi)_x$ in this basis and using \eqref{ide} gives the result.
\end{proof}

\begin{re}[Consistent diagonalization]\label{consist}
In smooth setting, according to \cite[Lemma 2.3]{pan}, at any point of a dense open subset of $M$, we have a local orthonormal frame of eigenvector fields $\{ U, E_1, E_2 \}$ of the Cauchy-Green tensor $\Cc_\varphi$, corresponding to the eigenvalues $\{ 0, \lambda_1^2, \lambda_2^2 \}$.

In this case the proof of Lemma \ref{taulem} can be obtained along the lines of \cite[Corollary 3.1]{slo}.
\end{re}

\begin{co}\label{eqs2}
For a $C^2$ mapping $\varphi: (M^3, g) \to (N^2, h)$, the $\Ee_{\sigma_2 , P}$-Euler-Langrange equations at regular points are equivalent with
\begin{equation}\label{s2pot}
\lambda_{1}^2 \lambda_{2}^2 \left(\mu^{\VV}  -\gr^{\HH}(\ln \abs{\lambda_{1}\lambda_{2}})\right) =  -\gr (\ov{P} \circ \varphi),
\end{equation}
and, if the $\Ee_{\sigma_2 , P}$-Euler-Langrange equations are satisfied everywhere on $M$, then the critical points of $\varphi$ are also critical points of $P=\ov{P} \circ \varphi$. 

In particular, if an almost submersion $\varphi$ has minimal fibres $($i.e. $\mu^{\VV}=0 )$ and $\frac{1}{2}\lambda_{1}^2\lambda_{2}^2=\ov P\circ\varphi$, then $\varphi$ is $\sigma_2$-critical with potential $\ov P$.
\end{co}

\begin{proof}
By Proposition \ref{sig2pot}, $\varphi$ is $\sigma_2$-critical with potential $\ov{P}$ iff $\tau_{\sigma_2}(\varphi)_x = \gr \ov{P}_{\varphi(x)}$ for any $x \in M$. At regular points, $x\in M\setminus M_\varphi$, using the same basis as in Lemma \ref{taulem}, we have $\gr \ov{P}_{\varphi(x)}=  
\frac{1}{\lambda_{1}^{2}} E_1(\ov{P} \circ \varphi)\dif \varphi(E_1)
+\frac{1}{\lambda_{2}^{2}} E_2(\ov{P} \circ \varphi)\dif \varphi(E_2)$,
which, combined with \eqref{taueigen}, yields Equation \eqref{s2pot}. 

Now let $x \in M_\varphi$ be a critical point, i.e. $\rank \dif \varphi_x \leq 1$. Then  $\varphi^* \omega_x=0$, so for any $X \in T_xM$ we have $X (\ov{P} \circ \varphi)_{x}=h(\gr \ov{P}_{\varphi(x)}, \dif \varphi_x(X))=-h(J\dif \varphi_x\big((\delta\varphi^* \omega)^\sharp \big), \dif \varphi_x(X))=\varphi^* \omega_x((\delta\varphi^* \omega)^\sharp, X)=0$, where the second equality is given by the Euler-Lagrange equations \eqref{om2poteqs}. This proves that $x$ is critical also for $P$, i.e. $\gr (\ov{P} \circ \varphi)_{x}=0$ for any $x \in M_\varphi$.

The second statement is immediate.
\end{proof}

\section{Field-Fluid Correspondence}
In this section we establish the correspondence between steady solutions of the Euler system \eqref{mhd} and field solutions of the $\sigma_2$-variational problem with potential. Then we point out some direct applications of this duality.

\begin{pr}\label{main}
If the $C^2$ mapping $\varphi:(M^3, g) \to (N^2, h)$ is $\sigma_2$-critical with potential $\ov{P}$, and $\omega$ is the area 2-form on $N$ induced by $h$, then the vector field $V = (\ast \varphi^* \omega)^\sharp$ satisfies the Euler equations for steady incompressible flows on $M$ with Bernoulli function $P=\ov{P} \circ \varphi$. Conversely if $V$ is a steady incompressible Euler solution on $M$, then it exists locally a $\sigma_2$-critical submersion with potential into some surface $(N,h)$ with fibres tangent to $V$.
\end{pr}

\begin{proof}
Let $\varphi$ be a $\sigma_2$-critical map with potential $\ov{P}$. Around a regular point $x$, let $U$ be the (locally defined) unit vertical vector field. Then the vector field $V = (\ast \varphi^* \omega)^\sharp$ can be rewritten as $V = \lambda_{1}\lambda_{2}U$.  Since $\mu^\VV = \nabla_U U$, Equation \eqref{s2pot} shows that $V$ satisfies the first fluid equation in \eqref{euler+}. The divergence-free condition for $V$, which reads $\di U + U(\ln \abs{\lambda_{1}\lambda_{2}})=0$, is also satisfied since it turns out to be equivalent to  $(\di S_{\sigma_2}(\varphi))(U)=0$, which is true by \eqref{ide}. Another way of proving this part is to simply compare the momentum-flux tensor \eqref{fluidstress} of the fluid (see its alternative expression) with the $\sigma_2$-stress-energy tensor \eqref{stress2} of the field, adjusted by the presence of the potential term, and then to apply \eqref{ide}.

At critical points $x \in M_\varphi$, by construction we have $V_x=0$ and by Corollary \ref{eqs2} we have $\gr P_x =0$, so the steady incompressible Euler equations are satisfied also at these points. 

Conversely, let $V$ be a solution of the equations \eqref{euler+} for some Bernoulli function $P$. In a neighbourhood of a point $x \in M$ where $V_x \neq 0$, let $\varphi$ denote the local projection along $V$ into some 2-manifold $N$. Denote by $\vartheta$ the 1-form dual to $V$. The continuity equation $\di V=0$ implies
$$\mathcal{L}_V \left(\ast \vartheta \right)=0,$$
so that $\ast \vartheta$ is a basic 2-form and it descends to an area 2-form $\omega$ on $N$. Choosing an associated metric $h$ for the symplectic structure $\omega$, we can check that $\abs{\varphi^*\omega}=\lambda_1 \lambda_2=\abs{V}$, where $\lambda_{1}^2$, $\lambda_{2}^2$ are the non-zero eigenvalues of $\varphi^* h$ with respect to $g$. Moreover since $V(P)=0$, $P$ is a projectable function: $P=\ov{P} \circ \varphi$ for some $\ov{P}$ defined on $N$. Therefore the Euler first equation in \eqref{euler+} translates into $\lambda_{1}^2 \lambda_{2}^2 \ T_\varphi =  -\gr (\ov{P} \circ \varphi)$ which, by Corollary \ref{eqs2}, implies that $\varphi$ is $\sigma_2$-critical with potential $\ov P$. 
\end{proof}

%Notice that if either $\rank \dif \varphi \in \{0, 1\}$, or if $\rank \dif \varphi$ is locally constant on a dense subset (as for smooth mappings), or $\varphi$ is an almost submersion we can strengthen the first assertion above: $\varphi:(M^3, g) \to (N^2, h)$ is $\sigma_2$-critical with potential $\ov{P}$, if and only if the vector field $V = (\ast \varphi^* \omega)^\sharp$ satisfies the Euler equations for steady incompressible flows on $M$ with Bernoulli function $P=\ov{P} \circ \varphi$.

We may see the above result either as a variational characterization of steady incompressible Euler flows (see \cite[p. 75]{arn} and \cite{mars} for an alternative) or as a new physical interpretation of the strongly coupled Faddeev-Skyrme model.

The correspondence between steady Euler flows and $\sigma_2$-critical submersions with potential is only local. Those steady flows for which the correspondence happens to be global will be called $S$-integrable, where $S$ is standing for 'submersion' or 'simple', since the 1-dimensional (singular) foliation given by the integral curves of $V$ (allowed to have zeros) is simple.

\begin{de}
A vector field $V$ on a manifold $M$ is called $S$-\emph{integrable} if at any regular point it is tangent to the fibres of a map $\varphi: M^m \to N^{m-1}$.
\end{de}

\subsection{Forced Euler flows and solutions of the full Faddeev-Skyrme model with mass term}
Including the standard Dirichlet density into the $\sigma_2$-energy functional yields ($\kappa \in \RR_+$):
\begin{equation}\label{si12pot}
\Ee_{\sigma_{1,2} , P}(\varphi)=\frac{1}{2}\int \{\kappa \abs{\dif \varphi}^2 + \abs{\wedge^2 \dif \varphi}^2 + 2 \ov{P} \circ \varphi \}\nu_g ,
\end{equation}
that gives rise to the corresponding notion of $\sigma_{1,2}$-\textit{critical mapping} (\cite{slo}) with potential, as in the \textit{full Faddeev-Skyrme model} with mass term.

By a completely analogous argument as above, using the stress-energy tensor
$S_{\sigma_1}(\varphi)=\frac{1}{2}\abs{\dif \varphi}^2 \, g -\varphi^*h$ associated to the Dirichlet energy, we obtain

\begin{pr}
If $\varphi: (M^3 , g) \to (N^2, h)$ is a $\sigma_{1,2}$-critical map with potential $\ov P$and $\omega$ is the area form on $N$, then $V=\ast \varphi^* \omega$ satisfies the \emph{steady forced Euler equations}:
$$\nabla_V V = - \gr p + F, \quad \di V = 0,$$
where $F=\kappa(\di \Cc_\varphi-\frac{1}{2}\gr\abs{\dif \varphi}^2)$ and $p=\ov{P}(\varphi)-\tfrac{1}{2}\abs{V}^2$. In particular, $F\perp V$.
\end{pr}

\subsection{Beltrami flows and $\sigma_2$-critical fields with finite energy} Proposition \ref{main} shows in particular that the "pure" $\sigma_2$-variational problem (with $\ov P \equiv$ cst.) for mappings with 1-dimensional fibres on $M^3$, gives rise to Beltrami flows.

By the standard \emph{Derrick's scaling argument} \cite{der}, the pure $\sigma_2$-energy functional on $\RR^3$ doesn't allow non-trivial critical points with finite energy. This can be now seen as a Liouville property for $S$-integrable Beltrami flows. This allows us to guess that for general Beltrami flows the same property will hold. Indeed, it has been recently proved (\cite{nad}) that: 
\begin{quote}
If a $C^1$ Beltrami field in $\RR^3$ has finite energy $\int_{\RR^3}\abs{V}^2 \nu_g$, then it is identically zero.
\end{quote}.

\subsection{Axisymmetric flows and the rational map ansatz}
The original Faddeev-Skyrme model is defined for fields $\varphi: \RR^3 \to \Ss^2$ satisfying the asymptotic boundary condition $\lim_{\abs{x}\to \infty}\varphi(x) =(0,0,1)$, $x=(x_1, x_2, x_3)\in \RR^3$. The latter allows the one-point compactification  of $\RR^3$ to $\Ss^3=\{(z_0, z_1) \in \CC^2 | \abs{z_0}^2 + \abs{z_1}^2=1\}$, given by $(z_0 , z_1) =$ $\big(\cos f(r)+ \mathrm{i} \frac{x_3}{r}\sin f(r) , \ \frac{x_1 + \mathrm{i} x_2}{r}\sin f(r)\big)$, where $f(0)=\pi$ and $f(\infty)=0$ (\cite{sut}). Therefore a particular form of $\varphi$ may be
$$
(z_0, z_1) \mapsto w=\frac{z_1^\ell}{z_0^k}
$$
where $w=(\varphi_1 + \mathrm{i}\varphi_2)/(1 + \varphi_3) \in \CC$ represents a point of $\Ss^2$ via the stereographic projection. This field ansatz is usually denoted by $\mathcal{A}_{k, \ell}$ and it is a particular realisation of an axially symmetric  field (in cylindrical coordinates $(\rho, \theta, z)$)
\begin{equation}\label{ax}
\varphi(\rho e^{\ii \theta}, \ z)= \big(\sin\Theta(\rho, z) \, e^{\ii (\ell \theta-k \psi(\rho, z))}, \ \cos\Theta(\rho, z)\big)
\end{equation}
where $k, \ell \in \ZZ$. 

\begin{re}
The axial symmetry is usually chosen since it is the maximal one allowing an arbitrary \emph{Hopf invariant} (\cite{ku}; see the definition of this invariant in $\S 4.4$). Notice that surfaces of constant $\Theta$ are homeomorphic to tori. Then, at least for solutions of the strongly coupled model with a potential depending only on $\varphi_3=\cos\Theta(\rho, z)$, this \textit{ansatz} choice is also a particular realization of the 2-dimensional foliation prescribed by the Arnold structural theorem applied to the associated steady Euler flow (cf. the discussion following \eqref{mhd}), which is collinear with $W=k(\partial_\rho \Theta \partial_z \psi - \partial_z \Theta \partial_\rho \psi)\, \partial_\theta-\ell (\partial_z \Theta \, \partial_\rho - \partial_\rho \Theta \, \partial_z) \in \ker \dif \varphi$. Notice that $W$ is an axially simmetric vector field; if moreover $\psi = A(\Theta)$, then its $\partial_\theta$-component vanishes and we say that $W$ (and $\varphi$) has \emph{no swirls}.
\end{re}

According to \cite[Theorem 5.3]{jiu}, there exists no non-trivial $C^1$ axisymmetric without swirls solution of the fluid equations \eqref{euler0} having finite energy $\int_{\RR^3}\abs{V}^2 \nu_g$ and satisfying $\displaystyle \lim_{\abs{x}\to \infty}\abs{V} = 0$ and $\displaystyle  \lim_{\abs{x}\to \infty}p = cst$. Using Proposition \ref{main} applied to $V=(\ast \varphi^* \omega)^\sharp$ corresponding to $\varphi$ in \eqref{ax}, and the above Remark, we get the following 

\begin{pr}\label{noexist}
There exists no finite energy axisymmetric without swirls $C^2$ solutions of the strongly coupled Faddeev-Skyrme model on $\RR^3$ with the potential of the form $P=(1-\varphi_3^a)^b$, $a, b \in \RR_+$.
\end{pr}
Recall that $(a, b)=(1,1)$ and $(a, b)=(2,1)$ are standard choices of the potential called \textit{old} and \textit{new baby  Skyrme potentials},  respectively.

\subsection{Topological energy bounds}
Due to the asymptotic boundary  condition $\lim_{\abs{x}\to \infty}\varphi(x) = cst.$, the field configurations in the Faddeev-Skyrme model may be seen topologically as maps  $\Ss^3 \to \Ss^2$.  Recall that for such maps the homotopy classes are indexed by the \textit{Hopf invariant} (or \textit{charge}) that can be computed by the integral formula: $Q(\varphi) = \tfrac{1}{16\pi^2}\int_{\Ss^3} \alpha \wedge \varphi^* \omega \in \pi_3(\Ss^2) \cong \ZZ$, where $\dif \alpha = \varphi^* \omega$, and $\omega$ is the standard area form on $\Ss^2$. The field configurations that minimize the energy in each homotopy class are the most interesting solutions for the Faddeev-Skyrme model, and they are called \textit{hopfions} (in analogy with \textit{skyrmions} for the Skyrme model; by extension any solution of Euler-Lagrange equations may be called in this way).

For null-homologous divergence-free fields $V$ on a compact, connected, oriented Riemannian 3-manifold $M$ with volume form $\nu_g$, a similar invariant (to the action of volume preserving diffeomorphisms) called \textit{helicity} is defined as $H(V)= \tfrac{1}{\pi^2}\int \alpha \wedge \imath_V \nu_g $, where $\dif \alpha = \imath_V \nu_g$ (see \cite[p.121-128]{arn} for more details). It is well known the fact that "helicity bounds energy"  (\cite{arn}):
$$
\int_{M}\abs{V}^2 \nu_g \geq \pi^2 \mu_1 \abs{H(V)},
$$
where $\mu_1>0$ is the first non-zero eigenvalue of the curl operator, and this yields the minimization property of principal eigenvectors of curl. One classical example of such field is the \textit{Hopf vector field} $\xi$ on $\Ss^3$.  Less surprisingly by now, the associated Hopf map was shown \cite{svee} to be a (global) minimizer for $\mathcal{E}_{\sigma_2}$ in its homotopy class. This result makes use of the analogous general lower bound for $\mathcal{E}_{\sigma_2}$ of algebraically inessential maps $\varphi:M \to \Ss^2$, in terms of their Hopf invariant (\cite{spe}): 
$$
\Ee_{\sigma_{2}}(\varphi) \geq 8\pi^2 \mu_1 \abs{Q(\varphi)},
$$
where $\mu_1^2 >0$ is the first non-zero eigenvalue of the Laplacian on coexact 1-forms on the domain; the equality is reached if $\ker \dif \varphi$ is spanned by a linear Beltrami field with $f \equiv \mu_1$.

For the strongly coupled Faddeev-Skyrme model with mass term on $\RR^3$ it has been recently shown \cite{har} that 
$$\Ee_{\sigma_{2} , P}(\varphi) \geq C \abs{Q(\varphi)}^{3/4},$$
where $C$ is a constant depending on $\ov P$. In view of the fluid-field duality, given the Bernoulli function, we obtain another topological bound for the energy $\int_{\RR^3}\abs{V}^2 \nu_g$ of a steady fluid on $\RR^3$ (vanishing at infinity) in terms of its helicity.

\section{Examples}

The aim of this section is to illustrate by explicit constructions the discussion in the previous section. Due to Proposition \ref{noexist} we have been led to consider examples on other spaces than $\RR^3$. The computational details are available in auxiliary files \cite{mathm}.

\begin{ex}
Let $\varphi : \mathbb{R}^2 \times \Ss^1 \to \Ss^2$ be defined by:
$$
\varphi(r e^{\ii \theta}, \ e^{\ii \phi})=\left(\sin\alpha(r)e^{\ii (\theta-\phi)}, \ \cos \alpha(r)\right)
$$
where $\theta, \phi \in [0, 2\pi]$, $r \geq 0$, $\alpha(r)=\arccos\left(1-\tfrac{2}{\sqrt{r^2 + 1}}\right)$ and the metrics considered are the standard ones. We can check that in this case $\abs{\lambda_1 \lambda_2}=\tfrac{2}{r^2 + 1}$ and the unit vector field tangent to the fibres is 
$$
U=\tfrac{1}{\sqrt{r^2 + 1}}(\tfrac{\partial}{\partial \theta}+\tfrac{\partial}{\partial \phi}).
$$

Take $\ov{P}(\varphi_1, \varphi_2, \varphi_3)=\tfrac{1}{16}(1-\varphi_3)^4$, so that
$P(r e^{\ii \theta}, \ e^{\ii \phi})= \tfrac{1}{(r^2 + 1)^2}$. 

Then by direct computation we can establish that
\begin{itemize}
\item  $V=\lambda_1 \lambda_2 U$ is a steady Euler flow with pressure\\ $p=P-\tfrac{1}{2}\abs{V}^2=-\tfrac{1}{(r^2 + 1)^2}$;

\item $\varphi$ is a $\sigma_2$-critical smooth submersion with potential $\ov P$. 
\end{itemize}
Notice moreover that $\lim_{\abs{(x,y)}\to \infty}\varphi(x,y,z)=(0,0,1)$ for any $z$, and that 
$\tfrac{1}{2}\int_{\mathbb{R}^2 \times \Ss^1} \abs{V}^2 \nu_g=
\Ee_{\sigma_2}(\varphi)=8 \pi^2$ (finite).
\end{ex}

It is interesting to notice that the above $V$ extends to a steady Euler flow on $\RR^3$ which is a Beltrami field with respect to a conformally flat metric (cf. \cite[p.134]{ble}). The potential appearing in this example has been considered also in \cite{lpz}.

For further analysis of (winding) Hopfions on $\mathbb{R}^2 \times \Ss^1$, see \cite{fosh, jay} (without potential) and \cite{nitt} (with potential).

\subsection{Examples on the 3-sphere} The Faddeev-Skyrme model (without potential) on $\Ss^3$ was first considered in \cite{wa}. In the following if no metric is specified, then the canonical round metric is considered.

The first standard (already mentioned) example is provided by the Hopf map on $\Ss^3$ which is a (smooth, unit charge) $\Ee_{\sigma_2}$-minimizer in its homotopy class and its associated Hopf vector field $\xi$ which is a ($S$-integrable, unit helicity) Beltrami field.

Steady Euler flows with higher helicities on $\Ss^3$ can be easily constructed by particularizing the class of solutions given in \cite[Example 4.5]{khes}. On the other side, there are several attempts to find higher Hopf charge solutions for the $\sigma_2$-variational problem: \cite{fer, slo, slobo} for the pure $\sigma_2$-energy case and \cite{ada, fos} for the case with potential (\cite{fos} provides numerical solutions). Nevertheless the existing higher charge field configurations we are aware of are \textit{not} classical solutions; they may be weak solutions (in some sense to be defined) or energy minimizers with low regularity. 

In the following we construct some classical solutions on $\Ss^3$ either by considering an adapted potential, or an adapted metric. We begin by recalling the construction in \cite{khes}. Consider $\Ss^3$ parametrized by
$$
(\cos s \cdot e^{\ii \phi_1}, \ \sin s \cdot e^{\ii \phi_2}), \quad s\in [0, \tfrac{\pi}{2}], \quad \phi_{1,2} \in [0, 2\pi],
$$
endowed with the standard round metric $g=\dif s^2 + \cos^2 \!s  \, \dif \phi_1^2 + \sin^2 \!s \, \dif \phi_2^2$. We stress that strictly speaking this parametrization correspond to a choice of local chart coordinates on $\Ss^3$ but we often deal with globally defined objects (when necessary this can be checked by translating  into Cartesian coordinates in $\RR^4$). Let $\xi_{-}=\tfrac{\partial}{\partial \phi_1}-\tfrac{\partial}{\partial \phi_2}$, $\xi_{+}=\tfrac{\partial}{\partial \phi_1}+\tfrac{\partial}{\partial \phi_2}$ be, respectively, the anti-Hopf and \textit{Hopf vector fields} on $\Ss^3$ (they are globally defined, smooth and non-singular). Define 
\begin{equation}\label{khes}
V=f_{-}(\cos^2 s)\xi_{-}+f_{+}(\cos^2 s)\xi_{+}
\end{equation}
with $f_{\pm}$ smooth functions. Then $V$ is a steady Euler flow with pressure $p(\cos^2 s)=2\int_{0}^{\cos^2 s} f_{-}(t)f_{+}(t)\dif t$. In order to have an $S$-integrable flow $V$ we assume $f_{+}+f_{-}=\ell h$ and $f_{+}-f_{-}=k h$, for $k, \ell \in \ZZ$ and $h(\cos^2 s)$ a smooth function. Then the associated submersion is of the type described below.

The $\alpha$-\textit{Hopf construction} (\cite{smi}) provides us with mappings $\varphi^{(k, \ell),\alpha}: \Ss^{3} \to \Ss^2(\tfrac{1}{2})$ defined by:
\begin{equation}\label{aho}
\varphi^{(k, \ell),\alpha}(\cos s \cdot e^{\ii \phi_1}, \ \sin s \cdot e^{\ii \phi_2}) = 
\left(\sin \alpha (s) \cdot e^{\ii (-k \phi_1 + \ell \phi_2)}, \  \cos \alpha (s) \right),
\end{equation}
where $k, \ell \in \ZZ^*$ and $\alpha:[0, \pi /2] \to [0, \pi]$ usually satisfies the boundary conditions $\alpha(\{0, \pi /2\})=\{0, \pi\}$. When $(k, \ell)=(\pm 1, 1)$ and $\alpha(s) = 2s$, this construction gives us the (conjugate) Hopf fibration as a Riemannian submersion. In general, the Hopf invariant is $Q(\varphi^{(k, \ell),\alpha})=k\ell$. Notice that the vector field $V_{k, \ell}=\ell\tfrac{\partial}{\partial \phi_1}+k\tfrac{\partial}{\partial \phi_2}$ is tangent to the fibres. 

Since generally $\varphi^{(k, \ell),\alpha}$ has critical points only at $s\in\{0, \pi /2\}$, we shall need the following 

\begin{re}[Regularity at the poles] Suppose that $\alpha$ is a smooth function and that it exists $a, b \geq 1$ for which the following limits are finite
$$\lim_{s\to 0} \frac{\sin \alpha (s)}{\sin^a s} \in \RR, \quad
\lim_{s\to \pi/2} \frac{\sin \alpha (s)}{\cos^b s} \in \RR. $$ 
If $a \geq \ell$ and $b\geq k$, then the mapping $\varphi^{(k, \ell),\alpha}$ is smooth on $\Ss^3$. 
\end{re}

\begin{ex}\label{exharm}
Let $\varphi=\varphi^{(k, k),\alpha} : \Ss^3 \to \Ss^2(\tfrac{1}{2})$ be defined by the $\alpha$-Hopf construction as above, thus having the Hopf invariant $Q(\varphi)=k^2$.
Choosing $\alpha(s)=2\arctan\left(\tan^k s \right)$ we obtain a globally defined,  smooth harmonic almost submersion with critical points located at $s=0$ and $s=\tfrac{\pi}{2}$. Note that this map is the composition of $z \to z^k$ with the Hopf map. We can check that in this case $\lambda_1^2 =\lambda_2^2$ and
$$\abs{\lambda_1 \lambda_2}=\frac{ k^2 \sin ^{2(k-1)}\!s \, \cos^{2(k-1)}\!s}{\left(\sin ^{2k} s + \cos ^{2k}s \right)^2}.$$

Define the charge-dependent potential:
$$
\ov P(\varphi_1, \varphi_2, \varphi_3) = \tfrac{k^4}{2^{5}}\left(1-\varphi_3^2\right)^{2(k-1)/k}\left((1+\varphi_3)^{1/k}+(1-\varphi_3)^{1/k}\right)^4
$$
For instance, if $k=2$, then $\ov P$ is given by the following linear combination of (generalized) new baby Skyrme potentials:\\ $\ov P(\varphi_1, \varphi_2, \varphi_3)=2\left(1-\varphi_3^2\right) +4\left(1-\varphi_3^2\right)^{3/2}+2\left(1-\varphi_3^2\right)^2$. 

The unit vector field tangent to the fibres of $\varphi$ is $U=\xi_{+}=\tfrac{\partial}{\partial \phi_1}+\tfrac{\partial}{\partial \phi_2}$ and one can see that the fibers are minimal $\mu^\VV=\nabla_U U =0$. By noticing the equality
$P(\cos s \cdot e^{\ii \phi_1}, \ \sin s \cdot e^{\ii \phi_2})=\tfrac{1}{2}\lambda_1^2 \lambda_2^2$, according to Corollary \ref{eqs2} we conclude that
\begin{itemize}
\item $V=\abs{\lambda_1 \lambda_2} U$ is a steady Euler flow with zero pressure;

\item $\varphi$ is a $\sigma_2$-critical smooth submersion with (charge dependent) potential $\ov P$. Being also harmonic, $\varphi$ is a solution on $\Ss^3$ of the full Faddeev-Skyrme model with potential $\ov P$.
\end{itemize}
Notice moreover that 
$$\int_{\Ss^3} \abs{V}^2 \nu_g=
\Ee_{\sigma_2,P}(\varphi)=\tfrac{2\pi ^2}{3}  \left(k^3+\left(k^2-1\right) \pi  \csc\left(\frac{\pi }{k}\right)\right)$$
which reduces to $\Ee_{\sigma_2,P}(\varphi)=2\pi^2$ for $k=1$ (i.e. Hopf map) case and has the following asymptotic behaviour $\lim_{k\to \infty}\tfrac{1}{k^3}\Ee_{\sigma_2,P}(\varphi)=\tfrac{4 \pi^2}{3}$ (i.e. for relatively large charges, $\Ee_{\sigma_2,P}(\varphi)$ scales as $Q(\varphi)^{3/2}$).
\end{ex}

For comparison let us derive one of the solutions in \cite{ada} starting with the steady Euler flow \eqref{khes} on $\Ss^3$ described above, again with $k=\ell$. In this case $f_{+}=h$ and $f_{-}=0$, so the pressure is constant. Therefore the Bernoulli function that will play the role of potential term is $P=\tfrac{1}{2}k^2 h^2(\cos^2 s) + cst$. For the associated submersion $\varphi^{(k,\ell), \alpha}$, asking that $\abs{\lambda_1 \lambda_2}=\abs{V}$ (cf. Proposition \ref{main}) and imposing the potential to be of the form $\ov P =1 - \varphi_3$ (old baby Skyrme potential) results in the equations:
\begin{eqnarray*}
&\alpha^\prime(s)\sin \alpha(s)=2\sin (2s)h(\cos^2 s); \\ 
&\alpha^\prime(s)\sin \alpha(s)=-k^2 \sin (2s)h(\cos^2 s)h^\prime(\cos^2 s)
\end{eqnarray*}
with the solution $\alpha(s)= \arccos\left(2(1+k^{-2})\cos^2 s -2k^{-2}\cos^4 s -1\right)$ that corresponds to \cite[(34)]{ada}. If $k=1$, this yields a smooth solution. If $k \geq 2$, then the submersion corresponding to $\alpha$ is continuous but not smooth (however, its energy is still well defined since its partial derivatives are continuous and bounded on the complement of a set of measure zero in $\Ss^3$).

\medskip

In Example \ref{exharm} we have seen that there exist classical solutions for our variational problem if we make a \textit{convenient choice of potential} depending on the homotopy class, a strategy reminiscent of \cite{lpz, war} in the baby Skyrme model case. In the last part of this section we shall keep standard choices for the potential but we \textit{allow non-standard  metrics} on the domain $M$ (again depending on the homotopy class). While taking this freedom is questionable from physical point of view (higher charge configurations might have a specific behaviour in approaching the vacuum or specific living space geometry), it may shed a new light to the existence theory for classical solutions for Faddeev-Skyrme model with mass term. Moreover let us recall that even for the (elliptic) harmonic map problem this kind of construction is the only known way to obtain (semiconformal) solutions between spheres (\cite[Ch.13]{ud}) in closed form.

The first non-standard metric that we consider in the next example is an ellipsoidal (or "squashed") metric, where the squashing factors are given in terms of $k$ and $\ell$ defining the Hopf invariant of the map. We note that this kind of squashed 3-sphere has been recently considered in the context of supersymmetric gauge theories (\cite{ham}).

\begin{ex}\label{squash1} Let $a >0$ and $k, \ell \in \ZZ$. Consider $\Ss^3$ endowed with the squashed metric
$$
g_{k, \ell} = a^2\left[(k^2 \sin^2 s + \ell^2 \cos^2 s)\dif s^2 + k^2 \cos^2 \!s \, \dif \phi_1^2 + \ell^2 \sin^2 \!s \, \dif \phi_2^2 \right]
$$
given by the restriction of the metric $G= a^2\left[k^2\abs{\dif z_0}^2 + \ell^2 \abs{\dif z_1}^2\right]$ on $\CC^2$. Let $\varphi=\varphi^{(k,\ell),\alpha}: (\Ss^3 , g_{k, \ell}) \to \Ss^2(\tfrac{1}{2})$ be defined by \eqref{aho}. Its $\sigma_2$-energy density can be directly computed to be
$$
\lambda_{1}^{2}\lambda_{2}^{2}=
\frac{\left[\alpha^{\prime}(s)\right]^2 \sin^2 \alpha(s)}{16 a^4 \sin^2 s \cos^2 s\left(k^2 \sin^2 s + \ell^2 \cos^2 s\right)}.
$$ 
The unit vector field tangent to the fibres of $\varphi$ is 
\begin{equation}\label{ukl}
U=\tfrac{1}{a k \ell}\left(\ell\tfrac{\partial}{\partial \phi_1}+k\tfrac{\partial}{\partial \phi_2}\right)
\end{equation}
and it can be seen that $\mu^\VV=\nabla_U U =0$. Choose the new baby Skyrme potential $\ov{P}(\varphi)=1-\varphi_3^2$. Then asking $\frac{1}{2}\lambda_{1}^2\lambda_{2}^2=\ov P(\varphi)$ (cf. Corollary \ref{eqs2}) we obtain $a^2=\tfrac{3\pi}{4\sqrt{2}} (k + \ell)/(k^2 + k \ell + \ell^2)$ and the profile function ($k\neq \ell$)
$$
\alpha(s)=\frac{\pi  \left(4 k^3-\sqrt{2} \left(k^2+\ell^2-\left(k^2-\ell^2\right)
\cos(2 s)\right)^{3/2}\right)}{4 \left(k^3-\ell^3\right)}
$$
satisfying boundary conditions $\alpha(0)=\pi$ and $\alpha(\pi/2)=0$. Notice that for $k=\ell$ we obtain $\alpha(s)=\pi \cos^2 s$ (compare with \cite[(55)]{ada}). The corresponding mapping $\varphi$ is an almost submersion with critical points generically located at $s=0$ and $s=\tfrac{\pi}{2}$, which is globally smooth for $\abs{k}, \abs{\ell} \in \{1,2\}$ and only $C^1$  otherwise (with the singular set of measure zero in the latter case).
We can conclude that
\begin{itemize}
\item $V=\abs{\lambda_1 \lambda_2} U$ is a steady Euler flow with zero pressure;

\item $\varphi$ is a $\sigma_2$-critical submersion (smooth for charges 1, 2 or 4) with standard potential $\ov{P}(\varphi)=1-\varphi_3^2$. 
\end{itemize}
Notice moreover that 
$\int_{\Ss^3} \abs{V}^2 \nu_g=\Ee_{\sigma_2,P}(\varphi)=2^{5/4} \pi ^{7/2} \sqrt{3} \, k \ell \sqrt{\frac{k+\ell}{k^2+k\ell+\ell^2}}$ which scales as $Q(\varphi)^{3/4}$ for $k=\ell$.
\end{ex}

The second non-standard metric is conformally related to the standard round metric.

\begin{ex}\label{squash2} Let $a >0$ and $k, \ell \in \ZZ$. Consider $\Ss^3$ endowed with the metric
$$
g_{k, \ell} = \frac{a^2 k^2 \ell^2}{k^2 \sin^2 s + \ell^2 \cos^2 s}\left(\dif s^2 +  \cos^2 \!s \, \dif \phi_1^2 +  \sin^2 \!s \, \dif \phi_2^2 \right).
$$
Let $\varphi=\varphi^{(k,\ell),\alpha}: (\Ss^3 , g_{k, \ell}) \to \Ss^2(\tfrac{1}{2})$ be defined by \eqref{aho} with he unit vector field tangent to the fibres given by \eqref{ukl}.
Again one has $\mu^\VV=\nabla_U U =0$. 

For the new baby Skyrme potential $\ov{P}(\varphi)=1-\varphi_3^2$, asking $\frac{1}{2}\lambda_{1}^2\lambda_{2}^2=\ov P(\varphi)$ (cf. Corollary \ref{eqs2}) gives us $a^2=\tfrac{\pi}{4\sqrt{2}} (k +\ell)/k\ell$ and the profile function ($k\neq \ell$)
$$
\alpha(s)=\frac{\ell \pi}{k-\ell}  \left(\frac{k\sqrt{2}}{[k^2+\ell^2-(k^2-\ell^2)\cos(2s)]^{1/2}} -1\right)
$$
satisfying $\alpha(0)=\pi$ and $\alpha(\pi/2)=0$. For $k=\ell$ we obtain again $\alpha(s)=\pi \cos^2 s$. The corresponding mapping $\varphi$ is an almost submersion with critical points at $s=0$ and $s=\tfrac{\pi}{2}$, which is globally smooth for $\abs{k}, \abs{\ell} \in \{1,2\}$ and only $C^1$  otherwise.
We can conclude that
\begin{itemize}
\item $V=\abs{\lambda_1 \lambda_2} U$ is a steady Euler flow with zero pressure;

\item $\varphi$ is a $\sigma_2$-critical submersion (smooth for charges 1, 2 or 4) with standard potential $\ov{P}(\varphi)=1-\varphi_3^2$. 
\end{itemize}
Notice moreover that 
$\int_{\Ss^3} \abs{V}^2 \nu_g=\Ee_{\sigma_2,P}(\varphi)=2^{-7/4} \pi ^{7/2} \sqrt{k\ell(k+\ell)}$ which scales as $Q(\varphi)^{3/4}$ for $k=\ell$.
\end{ex}

\begin{re} The vector field defined in Equation \eqref{ukl} is a unit (non-linear, rotational) Beltrami field with respect to both metrics considered in Examples \ref{squash1} ans \ref{squash2}.
\end{re}

Nevertheless the vector field \eqref{ukl} cannot be itself the steady Euler flow associated to an $\mathcal{E}_{\sigma_2}$-critical mapping into a surface (cf. Prop. \ref{main}, the $\sigma_2(\varphi)$ would be constant and $\varphi$ submersive, while the associated foliation has two singular orbits if $k \neq 1$, $\ell \neq 1$, a contradiction). In order to make this possible, one has to allow an orbifold codomain as in the following last example. This will be the appropriate realisation of the Boothby-Wang construction that we have attempted in \cite[Example 5.6]{slo} where neither the domain metric, nor the mapping $\varphi$ were globally smooth (in the present approach conical singularities are "confined" on the codomain).

\begin{ex}
The (smooth, non-singular, global) vector field $V_{k, \ell}=\ell\tfrac{\partial}{\partial \phi_1}+k\tfrac{\partial}{\partial \phi_2}$ is a linear Beltrami field with respect to the \emph{weighted Sasakian metric} $($\cite{taka}$)$ on $\Ss^3$ \begin{equation*}
g_{\mathbf{w}}=\tfrac{1}{\varsigma}\left(\dif s^2 + \tfrac{1}{\varsigma^2} \sin^2 \! s \cos^2 \! s(k\dif \phi_1 -\ell \dif \phi_2)^2\right) + \tfrac{1}{\varsigma^2}(\cos^2 \! s \, \dif \phi_1 + \sin^2 \! s \, \dif \phi_2)^2,
\end{equation*}
where $ \varsigma(\cos s e^{\ii \phi_1},\sin s e^{\ii \phi_2})= k \sin^2 \! s +  \ell \cos^2 \! s$ is a (globally defined) function on $\Ss^3$. In fact $V_{k, \ell}$ is the \emph{Reeb vector field} of the contact structure defined by $\eta_{\mathbf{w}}=\tfrac{1}{k\sin^2 s+\ell \cos^2 s}\eta_{can}$, where $\eta_{can}$ is the standard contact form on the round $3$-sphere, and $g_{\mathbf{w}}$ is an associated metric. The volume form of the weighted sphere and of the standard one are related by:
$$
\nu_{g_{\mathbf{w}}}=\frac{1}{\varsigma^2} \nu_{g_{can}},
$$
so in particular $\Vol_{g_{\mathbf{w}}}(\Ss^3)=\frac{2 \pi^2}{k\ell}$ (the higher charge configurations get more confined).

As steady Euler flow with respect to this new metric, $V_{k, \ell}$ is also the unit tangent vector field to the (geodesic) fibres of the smooth Riemannian orbifold submersion $\varphi(z_0, z_1) = [z_0^k, z_1^\ell]$ with $Q(\varphi)=k\ell$, which is therefore a pure $\sigma_2$-critical submersion from the weighted sphere $(\Ss^3, g_{\mathbf{w}})$ onto the \emph{weighted projective space} $\CC P^1_{\mathbf{w}}$ as compact complex orbifold with an induced K\"ahler structure $($\cite[Theorem 7.5.1]{boy}$)$. 
\end{ex}

\end{document}